\documentclass[12pt]{article}
\usepackage{}

\usepackage{epsfig}
\usepackage{latexsym}
\usepackage{caption}
\usepackage{amsfonts}
\usepackage{amssymb}
\usepackage{mathrsfs}
\usepackage{amsmath}
\usepackage{enumerate}
\usepackage{graphics}
\usepackage{MnSymbol}
\usepackage{float}
\usepackage{pict2e}
\usepackage{subfigure}

\usepackage{color}

\makeatletter

\newcommand{\Rmnum}[1]{\expandafter\@slowromancap\romannumeral #1@}

\makeatother

\begin{document}

\newtheorem{theorem}{Theorem}
\newtheorem{observation}[theorem]{Observation}
\newtheorem{corollary}{Corollary}
\newtheorem{algorithm}[theorem]{Algorithm}
\newtheorem{definition}{Definition}
\newtheorem{guess}[theorem]{Conjecture}
\newtheorem{claim}{Claim}
\newtheorem{problem}[theorem]{Problem}
\newtheorem{question}[theorem]{Question}
\newtheorem{lemma}{Lemma}
\newtheorem{proposition}[theorem]{Proposition}
\newtheorem{fact}[theorem]{Fact}

\makeatletter
  \newcommand\figcaption{\def\@captype{figure}\caption}
  \newcommand\tabcaption{\def\@captype{table}\caption}
\makeatother

\newtheorem{acknowledgement}[theorem]{Acknowledgement}

\newtheorem{axiom}[theorem]{Axiom}
\newtheorem{case}[theorem]{Case}
\newtheorem{conclusion}[theorem]{Conclusion}
\newtheorem{condition}[theorem]{Condition}
\newtheorem{conjecture}[theorem]{Conjecture}
\newtheorem{criterion}[theorem]{Criterion}
\newtheorem{example}[theorem]{Example}
\newtheorem{exercise}[theorem]{Exercise}
\newtheorem{notation}[theorem]{Notation}
\newtheorem{solution}[theorem]{Solution}
\newtheorem{summary}[theorem]{Summary}

\newenvironment{proof}{\noindent {\bf
Proof.}}{\rule{3mm}{3mm}\par\medskip}
\newcommand{\remark}{\medskip\par\noindent {\bf Remark.~~}}
\newcommand{\pp}{{\it p.}}
\newcommand{\de}{\em}
\newcommand{\mad}{\rm mad}
\newcommand{\qf}{Q({\cal F},s)}
\newcommand{\qff}{Q({\cal F}',s)}
\newcommand{\qfff}{Q({\cal F}'',s)}
\newcommand{\f}{{\cal F}}
\newcommand{\ff}{{\cal F}'}
\newcommand{\fff}{{\cal F}''}
\newcommand{\fs}{{\cal F},s}
\newcommand{\s}{\mathcal{S}}
\newcommand{\G}{\Gamma}
\newcommand{\g}{\gamma}
\newcommand{\wrt}{with respect to }

\newcommand{\qq}{\uppercase\expandafter{\romannumeral1}}
\newcommand{\qqq}{\uppercase\expandafter{\romannumeral2}}

\newcommand{\qed}{\hfill\rule{0.5em}{0.809em}}

\title{ Strong fractional choice number of series-parallel graphs}

\author{Xuer Li \thanks{Department of Mathematics, Zhejiang Normal University,  China. } \and Xuding Zhu\thanks{Department of Mathematics, Zhejiang Normal University,  China.  E-mail: xdzhu@zjnu.edu.cn. Grant Numbers: NSFC 11971438 and 111 project of Ministry of Education of China.}}

\maketitle

\begin{abstract}
	The strong fractional choice number of a graph $G$ is the infimum of those real numbers $r$ such that $G$ is 
	$(\lceil rm \rceil, m)$-choosable for every positive integer $m$. The strong fractional choice number of a family ${\cal G}$ of graphs is the supremum of the strong fractional choice number of graphs in ${\cal G}$.
		We denote by ${\cal{Q}}_k$ the class of series-parallel graphs with girth at least $k$. 
		This paper proves that for $k=4q-1, 4q,4q+1, 4q+2$, the strong fractional number of ${\cal{Q}}_k$ is exactly $2+ \frac{1}{q}$.
	

\end{abstract}

\section{Introduction}

A $b$-fold colouring of a graph $G$ is a mapping $\phi$ which assigns to each vertex $v$ of $G$ a set $\phi (v)$ of $b$ colours so that adjacent vertices receive disjoint colour sets. An $(a,b)$-colouring of $G$ is a $b$-fold colouring $\phi$ of $G$ such that $\phi(v) \subseteq \{1,2,\cdots,a\}$ for each vertex $v$. The {\em fractional chromatic number} of $G$ is
$$\chi_f(G)=\inf\{\frac a b :G \ \text{ is $(a,b)$-colourable}\}.$$

An \emph{ $a$-list assignment } of $G$ is a mapping $L$ which assigns to each vertex $v$ a set $L(v)$ of $a$ permissible colours. A \emph{$b$-fold $L$-colouring} of $G$ is a $b$-fold colouring $\phi$ of $G$ such that $\phi(v)\subseteq L(v)$ for each vertex $v$. We say $G$ is \emph{ $(a,b)$-choosable} if for any $a$-list assignment $L$ of $G$, there is a $b$-fold $L$-colouring of $G$. The {\em fractional choice number} of $G$ is $$ch_f(G)=\inf\{\frac a b :G \ \text{ is $(a,b)$-choosable} \}.$$

It was proved by Alon, Tuza and Voigt \cite{Chf} that for any finite graph $G$, $\chi_f(G)=ch_f(G)$ and moreover the infimum in the definition of $ch_f(G)$ is attained and hence can be replaced by minimum. This implies that if $G$ is $(a,b)$-colourable, then for some integer $m$, $G$ is $(am,bm)$-choosable. The integer $m$ depends on $G$ and is usually a large integer. A natural  question is for which $(a,b)$, $G$ is $(am,bm)$-choosable for any positive integer $m$. This motivated the definition of  strong fractional choice number  of a graph \cite{Zhu2018}.

\begin{definition}
	\label{def-strong}
	Assume $G$ is a graph and $r$ is a real number. We say $G$ is {\em strongly fractional $r$-choosable} if for any positive integer $m$, $G$ is $(\lceil rm \rceil, m)$-choosable. 
  The {\em strong fractional choice number} of $G$ is $$ch^*_f(G) = \inf\{r: \text{$G$ is strongly fractional $r$-choosable}\}.$$
   The strong fractional choice number of a class ${\cal G}$ of graphs is  $$ch^*_f({\cal G})= \sup \{ch^*_f(G):G \in {\cal G} \}.$$
\end{definition}

It follows from the definition that for any graph $G$, $ch_f^*(G) \ge ch(G)-1$. The variant $ch_f^*(G)$ is intended to be a refinement for $ch(G)$. However, currently we do not have a good upper bound for $ch_f^*(G)$ in terms of $ch(G)$. It was conjectured by Erd\H{o}s, 
Rubin and Taylor \cite{Erdos} that if $G$ is $k$-choosable, then $G$ is $(km,m)$-choosable for any positive integer $m$. If this conjecture were true, then we would have $ch_f^*(G) \le ch(G)$. But this conjecture is refuted recently by Dvo\v{r}\'{a}k, Hu and Sereni in \cite{Dvorak}. Nevertheless, it is possible that for any $k$-choosable graph $G$, for any positive integer $m$, $G$ is $(km+1,m)$-choosable. If this is true, we also have 
$ch_f^*(G) \le ch(G)$. In any case,   $ch_f^*(G)$ is an interesting 
graph invariant and there are many challenging problems concerning this parameter. The strong fractional choice number of planar graphs were studied in \cite{Zhu2018} and \cite{JZ2019}. Let ${\cal P}$ denote the family of planar graphs and for a positive integer $k$, let ${\cal P}_k$ be the family of planar graphs containing no cycles of length $k$.  It was proved in \cite{Zhu2018} that $5 \ge ch_f^*({\cal P}) \ge 4 + 2/9$ and prove in \cite{JZ2019} that $4 \ge ch_f^*({\cal P}_3) \ge 3+1/17$.

In this paper, we consider the strong fractional choice number of series-parallel graphs. For a positive integer $k$, let $${\cal Q}_k =\{G: \text{$G$ is a series-parallel graph with girth at least $k$} \}.$$ This paper proves the following result. 

\begin{theorem}
	\label{the1}
	Assume $q$ is a positive integer. For $k \in \{4q-1,4q,4q+1,4q+2\}$,  $ch^*_f({\cal Q}_k)=2+ \frac{1} {q}$.   
\end{theorem}

\section{The proof of Theorem \ref{the1}}

Series-parallel graphs is a well studied family of graphs and there are many equivalent definitions for this class of graphs. 
For the purpose of using induction, we adopt the definition that recursively construct  series-parallel graphs from $K_2$ by series parallel constructions.

\begin{definition}
	\label{def-sp}
	A \emph{two-terminal series-parallel graph} $(G;x,y)$ is defined recursively as follows:
	\begin{itemize}
		\item   Let $V(K_2)= \{0,1\}$.Then $(K_2;0,1)$ is a two-terminal series-parallel graph.
		\item (The parallel construction) Let $(G;x,y)$ and $(G';x',y')$ be two vertex disjoint two-terminal series-parallel graphs. Define $G''$ to be the graph obtained from the   union of $G$ and $G'$ by identifying $x$ and $x'$ into a single vertex $x''$, and identifying $y$ and $y'$ into a single vertex $y''$. Then $(G'';x'',y'')$ is a two-terminal series-parallel graph.
		\item    (The series construction) Let again $(G;x,y)$ and $(G';x',y')$ be two vertex  disjoint two-terminal series-parallel graphs. Define $G''$ to be the graph obtained from the union of $G$ and $G'$ by identifying $y$ and $x'$ into a single vertex $x''$. Then $(G'';x,y')$ is a two-terminal series-parallel graph.
	\end{itemize}
	A graph is a series-parallel graph if there exist some two vertices $x$, $y$ such that $(G;x,y)$ is a two-terminal series-parallel graph.
\end{definition}

\begin{lemma}
	\label{lem1}
Assume $m, l$ are positive integers, $\epsilon >0$ is a real number,  $P_{l}= (v_0,v_1,...,v_{l})$ is a path and 
$L$ is a list assignment of path $P_{l}$ with  $|L(v_0)|=m$, $|L(v_i)|=2m+\epsilon m$ ($1 \leq i \leq l$). For $0 \leq j \leq l$, there is a  subset $T_j$ of $L(v_j)$, for which the following holds:
	
		\begin{itemize}
			\item  If $j=2t+1$ is odd, then $|T_j|=m+\epsilon m$;
				If $j=2t$ is even, then $|T_j|=m$.
			\item 	For any $m$-subset $B_j$ of $L(v_j)$  for which $|B_j \cap T_j| \ge (1-t \epsilon)m$, there exists an $m$-fold $L$-colouring $\phi$ for $P_{j}=(v_0,v_1,...,v_j)$  such that $\phi(v_j)=B_j$.
		\end{itemize}	
\end{lemma}

\begin{proof}	
By induction on $j$. If $j=0$, then let $T_0=L(v_0)$; if $j=1$, then let $T_1=L(v_1)-T_0$. 
The conclusion is obviously true.

Assume $j \ge 2$ and the lemma holds for $j' < j$. 

\bigskip
\noindent
{\bf Case 1} $j=2t$ is even.

By induction hypothesis, there is an $(m+ \epsilon m)$-subset $T_{2t-1}$ of $L(v_{2t-1})$, such that 
for any  $m$-subset $B_{2t-1}$ of $L(v_{2t-1})$  for which $|B_{2t-1} \cap T_{2t-1}| \ge (1-(t-1) \epsilon)m$, there exists an $m$-fold $L$-colouring $\phi$ for $P_{2t-1}$  such that $\phi(v_{2t-1})=B_{2t-1}$.

As $|L(v_{2t})| = (2+\epsilon)m$, we have 
  $|L(v_{2t})-T_{2t-1}| \ge m$. Let $T_{2t}$ be any $m$-subset of $L(v_{2t})-T_{2t-1}$.
Assume $B_{2t}$ is an $m$-subset  of $L(v_{2t})$ with $|B_{2t} \cap T_{2t}| \ge (1-t \epsilon)m$. We shall show that   there exists an $m$-fold $L$-colouring $\phi$ for $P_{2t}$  such that $\phi(v_{2t})=B_{2t}$.

Note that $|B_{2t} \cap T_{2t-1}| \leq m-(1-t \epsilon )m=t \epsilon m$.
So   $|T_{2t-1}-B_{2t}| \ge  (1-(t-1) \epsilon )m$. Let $B_{2t-1}$ be an $m$-subset of $L(v_{2t-1})-B_{2t}$ containing 
at least $ (1-(t-1) \epsilon )m$ colours from $T_{2t-1}$.
By induction hypothesis, there exists an $m$-fold $L$-colouring $\phi$ of $P_{2t-1}$ with $\phi(v_{2t-1}) = B_{2t-1}$.
Now $\phi$ extends to an  $m$-fold $L$-colouring $\phi$ of $P_{2t}$ with $\phi(v_{2t}) = B_{2t}$.
 
\bigskip
\noindent
{\bf Case 2} $j=2t+1$ is odd.

By induction hypothesis, there is an $m$-subset $T_{2t}$ of $L(v_{2t})$, such that for  any  $m$-subset $B_{2t}$ of $L(v_{2t})$  for which $|B_{2t} \cap T_{2t}| \ge (1-t \epsilon)m$, there exists an $m$-fold $L$-colouring $\phi$ for $P_{2t}$  such that $\phi(v_{2t})=B_{2t}$.

As $|L(v_{2t+1})| = (2+\epsilon)m$, we have 
$|L(v_{2t+1})-T_{2t}| \ge (1+\epsilon)m$. Let $T_{2t+1}$ be any $(1+\epsilon)m$-subset of $L(v_{2t+1})-T_{2t}$.
Assume $B_{2t+1}$ is an $m$-subset  of $L(v_{2t+1})$ with $|B_{2t+1} \cap T_{2t+1}| \ge (1-t \epsilon)m$. We shall show that   there exists an $m$-fold $L$-colouring $\phi$ for $P_{2t+1}$  such that $\phi(v_{2t+1})=B_{2t+1}$.

Note that $|B_{2t+1} \cap T_{2t}| \leq m-(1-t \epsilon )m=t \epsilon m$.
So   $|T_{2t}-B_{2t+1}| \ge  (1-t \epsilon )m$. Let $B_{2t}$ be an $m$-subset of $L(v_{2t})-B_{2t+1}$ containing 
at least $ (1-t \epsilon )m$ colours from $T_{2t}$.
By induction hypothesis, there exists an $m$-fold $L$-colouring $\phi$ of $P_{2t}$ with $\phi(v_{2t}) = B_{2t}$.
Now $\phi$ extends to an  $m$-fold $L$-colouring $\phi$ of $P_{2t+1}$ with $\phi(v_{2t+1}) = B_{2t+1}$.
\end{proof}

\begin{corollary}
	\label{cor1}
	Assume $m$, $l$ are positive integers. Let 
	\[
	\epsilon=\begin{cases} \frac{2} {l-1}, & \text{if $l$ is odd} \\ \frac{2} {l}, &\text{if $l$ is even.} \end{cases}
	\]
	If $P_{l} = (v_0, v_1,..., v_l)$ is a path of length $l$ and $L$ is a list assignment of path $P_{l}$ with $|L(v_0)| = |L(v_l)| = m$, $|L(v_i)| = 2m+ \epsilon m$ for  $1\leq i \leq l-1$, then there is an $m$-fold $L$-colouring of $P_{l}$.
\end{corollary}	

\begin{proof}
	We divide the proof into two cases.
	
	\bigskip
	\noindent
	{\bf Case 1} $l=2t$ is even.

		By Lemma \ref{lem1},  there is an $(m+\epsilon m)$-subset $T_{l-1}$ of $L(v_{l-1})$, such that for any $m$-subset  $B_{l-1}$ of $L(v_{l-1})$, for which $|B_{l-1} \cap T_{l-1}| \ge (1-(t-1) \epsilon)m$, there exists an $m$-fold $L$-colouring $\phi$ for $P_{l-1}=(v_0,v_1,...v_{l-1})$, such that $\phi(v_{l-1})=B_{l-1}$.

		As  $t\epsilon m = m$, we have $|T_{l-1}-L(v_l)| \ge \epsilon m = (1-(t-1)\epsilon )m$.
		So there is an $m$-subset   $B_{l-1}$  of $L(v_{l-1})-L(v_l)$   containing at least $ (1-(t-1)\epsilon )m$ colours from $T_{l-1}$. 
		By Lemma \ref{lem1}, there exists an $m$-fold $L$-colouring $\phi$ of $P_{l-1}$ with $\phi(v_{l-1})=B_{l-1}$.
		Now $\phi$ can be extended to an $m$-fold $L$-colouring $\phi$ of $P_{l}$ with $\phi(v_{l})=L(v_l)$.

\bigskip
	\noindent
	{\bf Case 2} $l=2t+1$ is odd.
	
	Let $B$ be an $m$-subset of $L(v_{l-1})-L(v_l)$. 
	Let $L'(v_i)=L(v_i)$ for $1\le i \le l-2$ and $L'(v_{l-1}) =B$. 
	By Case 1, $P_{l-1} = (v_0,v_1,...v_{l-1})$ has an $m$-fold $L'$-colouring 
	$\phi$ with $\phi(v_{l-1})=B$.  
	Now $\phi$ can be extended to an $m$-fold $L$-colouring $\phi$ of $P_{l}$ with $\phi(v_{l})=L(v_l)$.
\end{proof}

\begin{lemma}
	\label{lem-girth}
	If $(G; x,y)$ is a series-parallel graph of girth $k$ and $l = \lceil k/2 \rceil$,
	then either $G$ itself is a path or $G$ contains a path $P=(v_0, v_1, \ldots, v_l)$ of length $l$ such that all the vertices 
	$v_1, v_2, \ldots, v_{l-1}$ are degree $2$ vertices of $G$, and none of them is a terminal vertex $x$ or $y$.
\end{lemma}
\begin{proof}
Assume $G$ contains a cycle $C$. If $G=C$, then the conclusion is true as $C$ has length at least $k$. 
Otherwise, $(G; x,y)$ is obtained from $(G_1; x_1,y_1)$ and $(G_2; x_2, y_2)$ by a series construction or a parallel construction. If one of   $(G_1; x_1,y_1)$ and $(G_2; x_2, y_2)$ contains a cycle, then $G_1$ or $G_2$ contains 
a required path. Otherwise, since $G$ contains a cycle, we conclude that $(G; x,y)$ is obtained from $(G_1; x_1,y_1)$ and $(G_2; x_2, y_2)$ by  a parallel construction, and for $i=1,2$, $G_i$ is a path connecting $x_i$ and $y_i$. Then $G$ is a cycle, and  the conclusion holds.
\end{proof}

\begin{theorem}
		\label{the2}
		Assume $q$, $m$ are positive integers, for any series-parallel graph $G$ with girth at least $k$, where $k\in \{4q-1,4q, 4q+1,4q+2\}$,  $G$ is $(\lceil(2+  \frac{1} {q}) m \rceil ,m)$-choosable.	
\end{theorem}

\begin{proof}
	Assume $L$ is a $\lceil(2+  \frac{1} {q}) m \rceil$-list assignment of $G$. 
	We need to show that $G$ has an $m$-fold $L$-colouring. 
	The proof is by induction on the number of vertices of $G$.
	If $G$ is a path, then $G$ is $(2m,m)$-choosable, and we are done. Assume $G$ is not a path.
	By Lemma \ref{lem-girth}, $G$ has a path  $P=(v_0, v_1, \ldots, v_l)$ of length $l$ (where $l=2q$ when $k \in \{4q-1, 4q\}$; $l=2q+1$ when $k \in \{4q+1, 4q+2\}$), 
	such that 
	all the vertices $v_1, v_2, \ldots, v_{l-1}$ are degree $2$ vertices of $G$. 
	Let $G'=G- \{v_1, v_2, \ldots, v_{l-1}\}$. Then $G'$ is also a series-parallel graph of girth at least $4q-1$ or $G'$ is a path. If $G'$ is a series-parallel graph of girth at least $4q-1$, then by induction hypothesis, $G'$ has an $m$-fold $L$-colouring $\phi$; If $G'$ is a path, as path is $(2m,m)$-choosable, so $G'$ also has an $m$-fold $L$-colouring $\phi$.
    Let $L'$ be the list assignment of the path $P=(v_0, v_1, \ldots, v_{l})$ with $L'(v_0) = \phi(v_0)$, $L'(v_{l}) = \phi(v_{l})$ and $L'(v_i) = L(v_i)$ for $1 \le i \le l-1$. By Corollary \ref{cor1}, $P$ has an $m$-fold $L'$-colouring $\psi$. Then the union of $\phi$ and $\psi$ is an $m$-fold $L$-colouring of $G$.  
\end{proof}

By Theorem \ref{the2}, for $k \in \{4q-1, 4q, 4q+1, 4q+2\}$, the strong fractional choice number of ${\cal{Q}}_k$ is at most $2+\frac{1} {q}$. In order to show that equality holds, we need to construct, for each positive integer $m$,  a graph belongs to ${\cal{Q}}_k$, which is not $(  (2+ \frac{1}{q} )m  -1 ,m)$-choosable.

\begin{lemma}
	\label{lem2}
Assume $m$, $l$ are positive integers and assume that $\epsilon$ is a positive real number 
such that   $\epsilon m$ is an integer and 
\[
\epsilon < \begin{cases} \frac{2} {l-1}, & \text{if $l$ is odd} \\ \frac{2} {l}, &\text{if $l$ is even.} \end{cases}
\]
Let	$P_{l}= (v_0,v_1,...,v_{l})$ be a path. Let $M_1$, $M_2$ be $m$-sets. Then there exists a list assignment $L$ of $P_{l}$ for which the following holds:
	\begin{itemize}
		\item   $L(v_0)=M_1$ and $L(v_{l})=M_2$.
		\item 	$|L(v_i)|=2m+\epsilon m$ ($1 \leq i \leq l-1$), 
		\item   there is no $m$-fold $L$-colouring of $P_{l}$.
	\end{itemize}		
\end{lemma}

\begin{proof}
Let $A_r$ (for $r=1,3,5,...,2q-3$), $B_s$ (for $s=2,4,6,...,2q-2$), $Z_t$ (for $t=1,3,5,...,2q-1$) be disjoint colour sets,   where $|A_r|=|B_s|=m$, $|Z_t|=  \epsilon m$.
Let $L$ be the list assignment of $P_{l}$ defined as follows:

\begin{itemize}
	\item   $L(v_0)=M_1$, $L(v_{l})=M_2$, $L(v_1)=M_1 \cup A_1 \cup Z_1$. 
	\item 	$|L(v_{2i+1})|=B_{2i} \cup A_{2i+1} \cup Z_{2i+1}$, for $i=1,2,3,...,q-2$.
	\item   $|L(v_{2j})|=B_{2j} \cup A_{2j-1} \cup Z_{2j-1}$, for $j=1,2,3,...,q-1$.
	\item If $l=2q$, then $L(v_{2q-1})=M_2 \cup B_{2q-2} \cup Z_{2q-1}$; if $l=2q+1$, then 
	$L(v_{2q})=M_2 \cup A_{2q-1} \cup Z_{2q-1}$.
\end{itemize}

We shall show that $P_{l}$ is not $L$-colourable.

\begin{claim}
	\label{claim1}
	For any $j \in \{2,3,4,...,q\}$, if $\phi$ is an $m$-fold $L$-colouring of $P_{2j-2}$, then 
	$|\phi(v_{2j-2}) \cap B_{2j-2}| \ge m-  (j-1) \epsilon m$. 
\end{claim}
\begin{proof}
	 We shall prove the claim by induction on the index $j$. 
	 Assume $j=2$ and $\phi$ is an $m$-fold $L$-colouring of $P_2$. As $\phi(v_0) = M_1$, we conclude that 
	 $\phi(v_1) \subseteq   A_1 \cup Z_1$. Therefore $|\phi(v_2) \cap ( A_1 \cup Z_1)| \le \epsilon m $.
	 Hence $|\phi(v_2) \cap B_2| \ge  m-\epsilon m$. 
	 
	 Assume $j \ge 3$ and the claim holds for $j' < j$ and $\phi$ is an $m$-fold $L$-colouring of $P_{2j-2}$.
	 Apply induction hypothesis to the restriction of $\phi$ to $P_{2j-4}$, we conclude that   
	 $$|\phi(v_{2j-4}) \cap B_{2j-4}| \ge  m- (j-2)\epsilon m.$$ 
	 Hence $|\phi(v_{2j-3}) \cap B_{2j-4}| \le  (j-2) \epsilon m$. 
	 This implies that  $$| \phi(v_{2j-3}) \cap (A_{2j-3} \cup Z_{2j-3})| \ge m- (j-2) \epsilon m .$$ 
	 Hence $$|\phi(v_{2j-2}) \cap (A_{2j-3} \cup Z_{2j-3})| \le (j-1) \epsilon m.$$ 
	 So $|\phi(v_{2j-2}) \cap B_{2j-2}| \ge m-  (j-1) \epsilon m$. 
\end{proof}

Assume $l=2q$ is even and $\phi$ is an $m$-fold $L$-colouring of $P_{2q}$. Then 
$|\phi(v_{2q-2}) \cap B_{2q-2}| \ge m-  (q-1) \epsilon m$.
As $\phi(v_{2q})=M_2$, we conclude that 
$\phi(v_{2q-1}) \subseteq (B_{2q-2} - \phi(v_{2q-2})) \cup Z_{2q-1}$. But 
$|(B_{2q-2} - \phi(v_{2q-2})) \cup Z_{2q-1}| \le (q-1) \epsilon m + \epsilon m  = q \epsilon m < m$,
a contradiction.

  Assume $l=2q+1$ is odd and $\phi$ is an $m$-fold $L$-colouring of $P_{2q+1}$. By claim \ref{claim1}, $|\phi(v_{2q-2}) \cap B_{2q-2}| \ge m-(q-1) \epsilon m$. Hence $|\phi(v_{2q-1}) \cap B_{2q-2}| \leq (q-1) \epsilon m $. This implies that $|\phi(v_{2q-1}) \cap (A_{2q-1} \cup Z_{2q-1})| \ge m-(q-1) \epsilon m$. As $\phi(v_{2q+1})=M_2$, we conclude that $\phi(v_{2q}) \subseteq (A_{2q-1} \cup Z_{2q-1})- \phi(v_{2q-1})$. But $|(A_{2q-1} \cup Z_{2q-1})- \phi(v_{2q-1})| \leq m+ \epsilon m -(m-(q-1) \epsilon m) = q \epsilon m < m$, a contradiction.
\end{proof}

For disjoint two-terminal series-parallel graphs $(G_1;l_1,r_1)$ and $(G_2;l_2,r_2)$, we use $G_1 \parallel G_2$ to denote the parallel composition of $G_1$ and $G_2$ and $G_1 \bullet  G_2$ to denote the series composition of $G_1$ and $G_2$, respectively. For a two-terminal series-parallel graph $G$, we let $G^{<n>}$ denote the series composition of $n$ copies of $G$, and let $G_{<n>}$ denote the parallel composition of $n$ copies of $G$.

\begin{theorem}
	\label{the3}
	Assume $m$, $q$ are positive integers and assume that $\epsilon$ is a positive real number such that $\epsilon m$ is an integer and $\epsilon < \frac{1}{q}$. For $k \in \{4q-1,4q, 4q+1,4q+2\}$, there exists a graph $G \in {\cal{Q}}_k$, such that $G$ is not $((2+ \epsilon )m, m)$-choosable.
\end{theorem}

\begin{proof}
Let $p={\binom{(2+ \epsilon )m}{m}}^2$.

Let graph $G$ be obtained by making parallel composition of $p$ paths with length $\left \lceil \frac{k}{2} \right \rceil$. Denote the two terminals by $x$ and $y$. Then $G \in {\cal{Q}}_k$.

We shall show that $G$ is not $((2+ \epsilon )m, m)$-choosable. Let $X$ and $Y$ be two $(2m+ \epsilon m)$-sets. Let $L(x)=X$ and $L(y)=Y$. There are $p$ possible $m$-fold $L$-colourings of $x$ and $y$. Each such a colouring $\phi$ corresponds to one path with length $\left \lceil \frac{k}{2} \right \rceil$. In that path, define the list assignment $L$ as in the proof of Lemma \ref{lem2}, by replacing $M_1$ with $\phi (x)$ and $M_2$ with $\phi (y)$. Then Lemma \ref{lem2} implies that no $m$-fold $L$-colouring of $x$ and $y$ can be extended to $G$. 
\end{proof}

By theorem \ref{the3}, for $k\in \{4q-1,4q, 4q+1,4q+2\}$, the strong fractional choice number of ${\cal{Q}}_k$ is at least $2+\frac{1} {q}$. Combining with theorem \ref{the2}, this completes the proof of theorem \ref{the1}.

\end{document}